\documentclass[letterpaper, conference]{ieeeconf}
\IEEEoverridecommandlockouts
\usepackage{cite}
\usepackage{lmodern}
\usepackage{amsmath}
\usepackage{graphicx}
\usepackage{bbm}
\usepackage[colorlinks=true, allcolors=black]{hyperref}
\usepackage{amssymb}
\usepackage{amsfonts}
\usepackage{colortbl}
\usepackage{dsfont}
\usepackage{bm}
\usepackage{algorithm}
\usepackage{algorithmic}
\usepackage{dsfont}
\usepackage[table, dvipsnames]{xcolor}
\usepackage{tabularx}
\usepackage{lipsum}
\usepackage{mathtools}
\usepackage{listings}
\allowdisplaybreaks

\newcommand{\espE}{\mathbb{E} }
\newtheorem{theorem}{Theorem}%[section]
\newtheorem{lemma}[theorem]{Lemma}
\newtheorem{proposition}[theorem]{Proposition}

\newcommand{\revAns}{\textcolor{black}}

\newcommand{\VBS}{V_{BS}}
\newcommand{\Veff}{V_{eff}}

\def\BibTeX{{\rm B\kern-.05em{\sc i\kern-.025em b}\kern-.08em
    T\kern-.1667em\lower.7ex\hbox{E}\kern-.125emX}}
\begin{document}

\title{Stabilization and Optimal Control of Interconnected SDE - Scalar PDE System}

% \author{\IEEEauthorblockN{Gabriel Velho}
% \IEEEauthorblockA{\textit{\small Université Paris-Saclay,} \\
% \textit{\small CentraleSupélec, CNRS} \\
% \textit{\small Laboratoire des Signaux} \\ %\textit{ autre chose}\\
% \textit{\small et Systèmes,} \\
% \small Gif-sur-Yvette, France \\
% \footnotesize gabriel.velho@centralesupelec.fr}
% \and
% \IEEEauthorblockN{Riccardo Bonalli}
% \IEEEauthorblockA{\textit{\small Université Paris-Saclay,} \\
% \textit{\small CentraleSupélec, CNRS} \\
% \textit{\small Laboratoire des Signaux} \\ %\textit{ autre chose}\\
% \textit{\small et Systèmes,} \\
% \small Gif-sur-Yvette, France \\
% \footnotesize riccardo.bonalli@centralesupelec.fr}
% \and
% \IEEEauthorblockN{Jean Auriol}
% \IEEEauthorblockA{\textit{\small Université Paris-Saclay,} \\
% \textit{\small CentraleSupélec, CNRS} \\
% \textit{\small Laboratoire des Signaux} \\ %\textit{ autre chose}\\
% \textit{\small et Systèmes,} \\
% \small Gif-sur-Yvette, France \\
% \footnotesize jean.auriol@centralesupelec.fr}
% \and
% \IEEEauthorblockN{Islam Boussaada}
% \IEEEauthorblockA{\textit{\small Université Paris-Saclay,} \\
% \textit{\small CentraleSupélec, CNRS} \\% \textcolor{red}{Need to add CNRS everywhere!!!}} \\
% \textit{\small Laboratoire des Signaux} \\ %\textit{ autre chose}\\
% \textit{\small et Systèmes,} \\
% \textit{\small Institut Polytechnique} \\
% \textit{\small des Sciences Avancées,} \\
% \small Gif-sur-Yvette, France \\
% \footnotesize islam.boussaada@centralesupelec.fr}
% }

\author{Gabriel Velho$^{1}$, % <-this % stops a space
\thanks{$^{1}$Universit\'e Paris-Saclay, CNRS, CentraleSup\'elec, Laboratoire des signaux et syst\`emes, 91190, Gif-sur-Yvette, France. } %{\tt\small gabriel.velho@l2s.centralesupelec.fr}.}%
Jean Auriol$^{1}$, % <-this % stops a space
%\thanks{$^{1}$Riccardo Bonalli is with Universit\'e Paris-Saclay, CNRS, CentraleSup\'elec, Laboratoire des signaux et syst\`emes, 91190, Gif-sur-Yvette, France. {\tt\small riccardo.bonalli@l2s.centralesupelec.fr}.}%
Riccardo Bonalli$^{1}$, % <-this % stops a space
%\thanks{$^{1}$Jean Auriol is with Universit\'e Paris-Saclay, CNRS, CentraleSup\'elec, Laboratoire des signaux et syst\`emes, 91190, Gif-sur-Yvette, France. {\tt\small jean.auriol@l2s.centralesupelec.fr}.}%
Islam Boussaada$^{2}$ % <-this % stops a space
\thanks{$^{2}$Universit\'e Paris-Saclay, CNRS, CentraleSup\'elec, Inria, Laboratoire des signaux et syst\`emes, 91190, Gif-sur-Yvette, France  \& IPSA, 94200, Ivry sur Seine, France.}%
}

\maketitle

%\addtolength{\abovedisplayskip}{-0.9mm} %en dessous et au dessus des équations
%\addtolength{\belowdisplayskip}{-2mm}

\begin{abstract}
In this paper, we design a controller for an interconnected system consisting of a linear Stochastic Differential Equation (SDE) actuated through a linear hyperbolic Partial Differential Equation (PDE). Our approach aims to minimize the variance of the state of the SDE component. We leverage a backstepping technique to transform the original PDE into an uncoupled stochastic PDE. As such, we reformulate our initial problem as the control of a delayed SDE with a non-deterministic drift. Under standard controllability assumptions, we design a controller steering the mean of the states to zero while keeping its covariance bounded. As final step, we address the optimal control of the delayed SDE employing Artstein's transformation and Linear Quadratic stochastic control techniques.
\end{abstract}

%In this paper, we design a controller for an interconnected system consisting of a linear Stochastic Differential Equation (SDE) actuated through a linear hyperbolic Partial Differential Equation (PDE). %One boundary of the PDE is actuated, while the other boundary is coupled to the SDE.
%Our approach aims to minimize the mean of the state within the SDE component. 
%In order to control the SDE in minimal time?
%We employ a backstepping technique to transform the original PDE into a non-coupled stochastic PDE. We then reformulate our initial problem as the control of a delayed SDE with a non-deterministic drift. We establish the existence of a minimal bound for the variance, achievable through a high gain controller, given certain assumptions on the SDE. Finally, we address the optimal control of the delayed SDE employing Artstein's transformation and Linear Quadratic stochastic control techniques.

%\begin{IEEEkeywords}
%Stochastic systems, Partial differential equations, Backstepping, Optimal control.
%\end{IEEEkeywords}

\section{Introduction}

The interest in interconnected systems of Ordinary Differential Equations (ODEs) and Partial Differential Equations (PDEs) emerged when delays in ODEs were associated with transport equations, allowing in~\cite{krstic_boundary_2008} for a re-interpretation of the classical Finite Spectrum Assignment \cite{artstein_linear_1982}. Interconnections involving hyperbolic PDEs and ODEs can model the propagation of torsional waves in drilling systems~\cite{aarsnes_torsional_2018}, deepwater construction vessels \cite{stensgaard_subsea_2010}, or heat exchangers systems.  Lyapunov and backstepping methods then facilitated the design of stabilizing controllers for such interconnected systems\cite{di_meglio_stabilization_2018, auriol_robustification_2023, deutscher_backstepping_2017, wang_delay-compensated_2020}. % \modifJA{Citer mon papier filtrage et Deutscher, et Wang} 
%and references therein. Specifically, exploration into stabilizing an ODE with a system of first-order linear hyperbolic PDEs in bidirectional coupling has been notable. This problem finds applications in various fields, such as reducing mechanical vibrations in drilling and modeling heat exchangers coupled with temperature systems. %Notably, in linear cases, backstepping techniques have proven effective in achieving robust stabilization and  trajectory tracking \cite{hu_control_2016}.

In realistic scenarios, dynamical processes are frequently influenced by disturbances originating, e.g., from imprecise measurements, parameter uncertainties, external disruptions, etc. \cite{yong_stochastic_1999}. Such disturbances may considerably alter the dynamics. Therefore, effectively mitigating these uncertainties is crucial to establish the reliability and safety of such controlled systems  \cite{ chen_optimal_2016}. Stochastic Differential Equations (SDEs) provide broad and accurate modelization of a large class of uncertain systems \cite{touzi_introduction_2013}. Stochastic control enables the effective design of stabilizing controllers for SDEs, which are also robust against random fluctuations. Stabilizing SDEs in expectation is a popular and effective control technique. However, when doing so, one must make sure the variance remains bounded to ensure reliability \cite{ chen_optimal_2016}. Robustness may also be reliably achieved by seeking controllers that keep the state variance relatively small, see \cite{ chen_optimal_2016 ,liu_optimal_2023}. %\modifJA{Parler de stabilisation de la moyenne?}

In interconnected ODE-PDE systems, modeling uncertainties that appear in the ODE dynamics is often key to accuracy and robustness. Some methods have been proposed to stabilize PDE+ODE systems in the presence of bounded or model-based disturbances \cite{redaud_output_2024, deutscher_backstepping_2017}. It has been shown that SDEs may offer a broader framework for modeling such disturbances.
Numerous methods exist for addressing the issue of SDEs with delays \cite{oksendal_maximum_2001, velho_mean-covariance_2023}. While delays can be interpreted as transport PDEs, coupled systems with more complexity have yet to be thoroughly investigated.
%Despite the need for efficient and robust control methods for coupled SDE+PDE systems, such research is lacking. 
Our study aims to provide insights and methodologies for effectively controlling such interconnected SDE-PDE systems for the first time.
%, which have not been thoroughly explored. %\modifJA{Il faut parler de notre papier ECC}
Specifically, in this paper, we bridge the gap between deterministic PDE+ODE control and stochastic control by merging methodologies from both fields. Our main contribution is the introduction of a novel approach to reliably control a PDE+SDE system with bidirectional coupling. The proposed controllers are implementable and straightforward to compute. 
 Concretely our contribution is twofold:
 \begin{itemize}
     \item We propose a feedback law achieving the stabilization of the PDE+SDE system, driving the mean of the states to zero while bounding the variance.
     \item Additionally, we introduce an optimal control approach to minimize the SDE variance over a finite time horizon, thus enhancing the robustness of the control strategy with respect to disturbances.
 \end{itemize}
 Our approach leverages the backstepping methodology to allow a reformulation of the coupled system as a system of delayed SDE with non-deterministic drift. We address our control objectives by employing Artstein's transformation and Linear Quadratic (LQ) control techniques.
%\modifJA{Parler de l'approche (place?)}

 The paper is organized as follows. In Section \ref{PF}, we outline the problem formulation. Next, in Section \ref{SR}, we utilize the backstepping transformation to reduce the problem to controlling an input delayed SDE with a random drift term, while also establishing the well-posedness of the system. In Section \ref{VC}, we propose a stabilizing feedback law for the SDE. Finally, in Section \ref{LQ}, we present an optimal control approach to minimize the variance of the SDE.

\section{Problem formulation}\label{PF}

\subsection{Notations}
For a given $n \in \mathbb{N} \setminus \{ 0 \}$, state variables take values in $\mathbb{R}^n$, while control variables take values in $\mathbb{R}$. 
We assume we are given a filtered probability space $(\Omega, \mathcal{F} \triangleq (\mathcal{F}_t)_{t \in [0,\infty)}, \mathbb{P})$, %For the sake of clarity in the exposition and without loss of generality, from now on, 
and that stochastic perturbations are due to a one-dimensional Wiener process $W_t$, which is adapted to the filtration $\mathcal{F}$.
%\textcolor{blue}{Comme dans \cite{yong_stochastic_1999}}. %We denote $\mathcal{F}_t$ the filtration induced by $W_t$.
Let $T > 0$ be some given time horizon. For any $r \in \mathbb{N} \backslash \{0\}$, we denote by $L_\mathcal{F}^2([0,T] , \mathbb{R}^r)$ the set of square integrable processes $P: [0,T]\times\Omega \to \mathbb{R}^r$ that are $\mathcal{F}$--progressively measurable, whereas the subset $C^2_\mathcal{F}([0,T] , \mathbb{R}^r) \subseteq L_\mathcal{F}^2([0,T] , \mathbb{R}^r)$ contains processes whose sample paths are continuous.
The spaces of semi-definite and definite positive symmetric matrices in $\mathbb{R}^n$ are denoted by $\mathcal{S}^+_n$ and $\mathcal{S}^{++}_n$, respectively. 
If $X \in L_\mathcal{F}^2([0,T] , \mathbb{R}^n)$, we denote by $V_X(\cdot)$ its variance, that is
$
V_X(t) \triangleq \espE[ (X(t) - \espE[X(t)])^T (X(t) - \espE[X(t)]) ] \in \mathbb{R}.
 $ Finally, we recall that if $f_1$ and $f_2$ are two deterministic functions in $L^2([0,T], \mathbb{R})$, It\^o formula yields %then ~\cite{le_gall_brownian_2016} 
 \small
 \begin{align}
     \espE \left[  \left( \int_0^t f_1(s) dW_s  \right) \left( \int_0^t f_2(s) dW_s  \right) \right] = \int_0^t f_1(s) f_2(s) ds. \label{eq:quadratic_variation_of_stochastic_integral}
 \end{align}
\normalsize

\subsection{Control system}

In this paper, we consider coupled PDE+SDE systems of the form
\begin{equation}\label{eq:coupled_PDE_SDE}
\left\{
\begin{array}{l}
     dX(t) = ( A X(t) + B v(t,0) ) dt + \sigma(t) dW_t \\
     u_t(t,x) + \lambda u_x(t,x) =  \eta^{+}(x) v(t,x) \\
     v_t(t,x) - \mu v_x(t,x) = \eta^{-}(x) u(t,x) \\
     u(t,0) = q v(t,0) + M X(t) \\
     v(t,1) = \rho u(t,1) + V_{in}(t) \\
     X(0) = X_0,~ u(0,x) = u_0(x),~ v(0,x) = v_0(x),
\end{array} 
\right.
\end{equation}
in the time-space domain $[0, +\infty) \times [0,1]$. The state of the system is $(X(\cdot), u(\cdot,x), v(\cdot,x)) \in \mathbb{R}^n \times (L^2([0,1]))^2$. The velocities $\mu>0$ and $\lambda>0$ are assumed to be constant, whereas the coupling terms $\eta^{+}$ and $\eta^{-}$ are continuous functions. The boundary coupling terms $\rho$ and $q$ verify $|\rho q|<1$ to avoid an infinite number of unstable poles, which is necessary to guarantee the existence of robustness margins for the closed-loop system \cite{auriol_delay-robust_2018}. \revAns{We also assume $q \neq 0$.} The matrices $A \in \mathbb{R}^{n\times n}, B \in \mathbb{R}^{n\times 1}, M \in \mathbb{R}^{1\times n}$ are constant, and $\sigma$ is a deterministic diffusion in $L^\infty([0,T], \mathbb{R}^n)$. Finally, $(X_0, u_0, v_0) \in \mathbb{R}^n\times (L^2([0,1]))^2$ and the control input $V_{in}$ takes values in $\mathbb{R}$.
% and we assume that the final time $T$ verifies $T > \frac{1}{\mu}$. 
%\modifJA{J'ai modifié l'ordre de présentation. On en discutera }
%\textcolor{blue}{Donner exemple de systeme décrit par cette eq?}
%\modifJA{Oui}
%\textcolor{blue}{Faire un schéma? }
%\modifJA{Pas nécessaire.}
% \modifJA{The class of system~\eqref{eq:coupled_PDE_SDE} naturally appears when modeling XXXX. We will show the well-posedness of the closed-loop system~\eqref{eq:coupled_PDE_SDE} for a given class of control input $V_{in}$ in Section \ref{SR}.} 
The class of system~\eqref{eq:coupled_PDE_SDE} naturally appear when modeling a heat-exchanger connected to a temperature system subject to random perturbations (e.g., the temperature of a building disrupted by a random outside temperature or sunlight). We defer the proof of the closed-loop well-posedness of system~\eqref{eq:coupled_PDE_SDE} to Section~\ref{SR}.

\subsection{Objective and approach overview}

%\modifJA{D'abord dire en terme "physique" ce qu'on souhaiterait faire. Est-ce qu'on stabilise l'ODE en moyenne?}
%The system naturally contains multiple feedback loops or couplings, which can potentially introduce instabilities. As a result, our initial objective is to devise a
The system naturally contains multiple feedback loops or couplings that can potentially introduce instabilities. The objective of the paper is twofold:
\begin{itemize}
    \item First, we aim to design a feedback control law to stabilize the mean of the system, driving it to zero while ensuring the variance remains bounded.
    \item Then, we provide a more noise-robust controller for the SDE state. This can be typically obtained through variance minimization \cite{liu_optimal_2023}. %We, therefore search for a controller minimizing the weighted quadratic form     \begin{equation}\label{eq:quadratic_form_minimized_objective}
   % \underset{V_{in} \in \mathcal{U}}{\min} \espE[ \int_0^T X(t)^T Q(t) X(t) dt ], \quad X \ \text{follows \eqref{eq:coupled_PDE_SDE}}.
%\end{equation}
\end{itemize}
%where $Q$ is a deterministic weight matrix in $\mathcal{S}^+_n$.

%Our goal is to minimize throughout the trajectory the variance of the state of the SDE $X$ in system \eqref{eq:coupled_PDE_SDE} through a feedback controller. To do so, we focus on the minimization of the weighted quadratic form 
%\begin{equation}\label{eq:quadratic_form_minimized_objective}
%    \underset{V_{in} \in \mathcal{U}}{\min} \espE[ \int_0^T X(t)^T Q(t) X(t) dt ], \quad X \ \text{follows \eqref{eq:coupled_PDE_SDE}}.
%\end{equation}
%\modifJA{En fait, ce n'est pas tout à fait ce qu'on résoud. Je ne pense pas qu'il faille donner les équations ici, mais uniquement en section 4.}
%where $Q$ is a deterministic weight matrix in $\mathcal{S}^+_n$. \modifJA{Dire que c'est un objectif classique. Donner une ref.}

%Minimizing directly the SDE state of system \eqref{eq:coupled_PDE_SDE} is a challenging task due to the PDE coupling: the control input can only act on the SDE system through the (possibly unstable) PDE system. \modifJA{Plutôt dire qu'on ne va pas s'intéresser outre mesure au système EDP en B.F. mais seulement garantir qu'il reste borné}\textcolor{blue}{Reférence minimisation EDP?}. 
To achieve the aforementioned design, we propose the following methodology: %The system naturally contains multiple feedback loops or couplings, which can potentially introduce instabilities. As a result, our initial objective is to devise a
\begin{enumerate}
    \item First, we use an invertible backstepping transformation to map the coupled SDE-PDE system into a cascade SPDE-SDE system. %This transformation moves the (possibly) unstable coupling terms $\eta^+$ and $\eta^-$ at the actuated boundary 
    %\modifJA{Le terme SPDE est-il déjà défini?}
    \item Secondly, using the method of characteristics, we prove the well-posedness of this latter system (and consequently, of the original system). We also show that the state $X$ can be expressed as the solution of an SDE with input delay and random drifts.
    \item  Then, using the Artstein transform, we design an appropriate stabilizing control law. 
    \item Finally, we tackle the minimization of the variance using  classical tools from stochastic LQ control.
\end{enumerate}
%begin by transforming the coupled PDE into a non-coupled SPDE. Subsequently, employing the method of characteristics, we resolve the transformed equation. We then show that the state $X$ follows an equivalent SDE with input delay and a random drift. This allows us to tackle the minimization of \eqref{eq:quadratic_form_minimized_objective} with the Artstein transform and classical tools from stochastic LQ control.

In our approach, the control law $V_{in}$ will write $V_{in} = \VBS + \Veff.$ 
\revAns{Here, the component $\VBS$ corresponds to a backstepping controller that would stabilize the PDE subsystem in the absence of the SDE. The component $\Veff$ corresponds instead to the controller that will be leveraged to solve the variance minimization problem. }
%enters the SDE with a delay. 

%We design the controller $\VBS$ with the backstepping approach by making it track the signal $\Veff$ in minimum time \cite{auriol_minimum_2016}. The controller $\Veff$ will in turn be designed to solve the input delayed stochastic LQ problem. 
%\modifJA{Ce que j'ai effacé n'était pas correct.}
%\textcolor{green}{Expliquer le process utilisé, possiblement un schéma qui explique quelles sont les étapes de résolution. }

\section{Backstepping Transformation for System Simplification.}\label{SR}

%\textcolor{orange}{Changer titre : time-optimal control/stabilization/signal traking in the PDE}

%In this section, we use the backstepping transformation design the controller $\VBS$ and reduce system \eqref{eq:coupled_PDE_SDE} into a delayed SDE with a random drift.
%\modifJA{Changer titre section}

In this section, we show that the original system~\eqref{eq:coupled_PDE_SDE} can be rewritten as a delayed SDE with random drifts. 

\subsection{Backstepping transformation}

%The backstepping change of variables in the PDE can be written as follows
Taking inspiration from \cite{auriol_delay-robust_2018}, we consider the following backstepping change of variables:
%\textcolor{orange}{Changer titre : the following change of variables inspired by REF papier Jean / Florent}
\begin{align}
    \begin{pmatrix}
        \alpha(t,x)\\\beta(t,x)
    \end{pmatrix}=\begin{pmatrix}
        u(t,x)\\v(t,x)
    \end{pmatrix}&+\int_0^x K(x,y)\begin{pmatrix}
        u(t,y)\\v(t,y)
    \end{pmatrix}dy \nonumber \\   &+\gamma(x)X(t),\label{eq:backstepping_transform_alpha_beta_definition}
\end{align}
% \begin{equation}\label{eq:backstepping_transform_alpha_beta_definition}
% \begin{split}
% \alpha(t,x) \triangleq u(t,x) & + \int_0^x K_{uu}(x,y) u(t,y) dy \\ 
% & + \int_0^x K_{uv}(x,y) v(t,y) dy + \gamma_\alpha(x) X(t) \\
% \beta(t,x) \triangleq v(t,x) & + \int_0^x K_{vu}(x,y) u(t,y) dy \\ 
% & + \int_0^x K_{vv}(x,y) v(t,y) dy + \gamma_\beta(x) X(t)
% \end{split}
% \end{equation}
where, $K(x,y)=\begin{pmatrix}
    K_{uu}(x,y) & K_{uv}(x,y)\\  K_{vu}(x,y) &  K_{vv}(x,y)
\end{pmatrix}$ and $\gamma(x)=\begin{pmatrix}
    \gamma_\alpha(x) & \gamma_\beta(x)
\end{pmatrix}^T$. The kernels  $K_{uu}, K_{uv}, K_{vu}$, and $K_{vv}$ are continuous functions defined on the triangle domain $\mathcal{T} \triangleq \left\{ (x,y), x \in [0,1], y \in [0,x]  \right\}$, whereas $\gamma_\alpha$ and $\gamma_\beta$ are in $C^1([0,1])$. On their respective domains of definition, $\gamma_\alpha$ and $\gamma_\beta$ verify the following set of \textit{kernel equations}:
%\textcolor{orange}{The kernels verify the following set of equations}
%\textcolor{red}{Donner ici les 2 sets d'équations de noyaux, changer Q et R en q et $\rho$.
%Changer l'ordre EDP -> Conditions -> ODE gamma}
%\textcolor{orange}{Ensuite well posedness des équations de noyaux}
%\textcolor{red}{Enfin : Differentiating the new variables gives the decoupled target system ...}
\begin{equation}\label{eq:kernel_equations_alpha}
\left\{
\begin{array}{l}
     \Lambda K_x(x,y)+K_y(x,y)\Lambda+K(x,y)\eta(y)=0, \\
     \Lambda\gamma\revAns{'}(x)+\gamma(x)\bar A+\lambda K(x,0)\bar M=0, \\
     \Lambda K(x,x)-K(x,x)\Lambda=-\eta(x), ~\gamma_\alpha(0)=-M,\\
     K(x,0)\begin{pmatrix}
         \lambda q & -\mu
\end{pmatrix}^T+\gamma(x)B=0,~\gamma_\beta(0)=0,
\end{array} 
\right.
\end{equation}
where $\Lambda=\text{diag}(\lambda, -\mu)$, $\eta=\begin{pmatrix}
    0& \eta^+\\\eta^- &0
\end{pmatrix}$, $\bar A=\text{diag}(A, A)$, $B=\text{diag}(B, B)$ and $\bar M=\text{diag}(M, 0_{1\times n})$.
% \begin{equation}\label{eq:kernel_equations_alpha}
% \left\{
% \begin{array}{l}
%      \lambda (K_{uu})_x(x,y) + (K_{uu})_y(x,y) \lambda + K_{uv}(x,y) \eta^{-}(y) = 0 \\
%      \lambda (K_{uv})_x(x,y) - (K_{uv})_y(x,y) \mu + K_{uu}(x,y) \eta^{+}(y) = 0 \\
%      \lambda  K_{uv}(x,x) + K_{uv}(x,x) \mu + \eta^{+}(x) = 0 \\
%      K_{uu}(x,0) \lambda q - K_{uv}(x,0) \mu + \gamma_\alpha(x) B = 0 \\
%      \lambda \gamma_\alpha'(x) + \gamma_\alpha(x) A + K_{uu}(x,0) \lambda M  = 0 \\
%      \gamma_\alpha(0) = -M,
% \end{array} 
% \right.
% \end{equation}
% as well as
% \begin{equation}\label{eq:kernel_equations_beta}
% \left\{
% \begin{array}{l}
%      - \mu (K_{vu})_x(x,y) + (K_{vu})_y(x,y) \lambda + K_{vv}(x,y) \eta^{-}(y) = 0 \\
%      - \mu (K_{vv})_x(x,y) - (K_{vv})_y(x,y) \mu + K_{vu}(x,y) \eta^{+}(y) = 0 \\
%      - \mu  K_{vu}(x,x) - K_{vu}(x,x) \lambda + \eta^{-}(x) = 0 \\
%      K_{vu}(x,0) \lambda q - K_{vv}(x,0) \mu + \gamma_\beta(x) B = 0 \\
%      - \mu \gamma_\beta'(x) + \gamma_\beta(x) A  + K_{vu}(x,0) \lambda M = 0 \\
%      \gamma_\beta(0) = 0.
% \end{array} 
% \right.
% \end{equation}
The set of equations \eqref{eq:kernel_equations_alpha} is well-posed and admits a unique solution $(K_{uu}, K_{uv}, K_{vu}, K_{vv}, \gamma_\alpha, \gamma_\beta)$ in $\left( C^0(\mathcal{T}) \right)^4 \times \left( C^1([0,1]) \right)^2$, see, e.g., \cite{auriol_delay-robust_2018}. %\modifJA{Donner la régularité. Citer plutôt mon article ODE-EDP}. 
Differentiating equation~\eqref{eq:backstepping_transform_alpha_beta_definition} with respect to time and space (in the sense of distributions), and integrating by parts, we can show that the target states $(\alpha,\beta,X)$ are solutions to the following \textit{target system}:
% We can now establish that the new variables $\alpha$ and $\beta$, verify the system of equations of a unidirectional coupled SPDE+SDE, now refered to as the \textit{target system}.
% \begin{lemma}\label{lem:target_system_alpha_beta}
%     The functions $\alpha$, $\beta$ and $X$ verify the following system of equations
\begin{equation}\label{eq:first_target_system}
\left\{
\begin{array}{l}
     dX(t) = ( A X(t) + B \beta(t,0) ) dt + \sigma(t) dW_t, \\
     d\alpha(t,x) + \lambda \alpha_x(t,x) dt = \gamma_\alpha(x) \sigma(t) dW_t, \\
     d\beta(t,x) - \mu \beta_x(t,x) dt = \gamma_\beta(x) \sigma(t) dW_t, \\
     \alpha(t,0) = Q \beta(t,0),~\beta(t,1) = \Veff(t),

     %X(0) = X_0, \quad \alpha(0,x) = \alpha_0(x), \quad \beta(0,x) = \beta_0(x).
\end{array} 
\right.
\end{equation}
where $\VBS$ is defined as 
% \end{lemma}
% \begin{proof}
% By differentiating $\alpha$ and $\beta$ in the sense of distributions, we find that they follow the weak formulation of the SPDE in \eqref{eq:first_target_system}, as defined in \cite{hairer_introduction_2009}. 
% \end{proof}
% We can see in the system \eqref{eq:first_target_system} that by setting 
\begin{equation}\label{eq:feedback_law_V_BS_PDE_target}
\begin{split}
    & \VBS(t)  = - \rho u(t,1) - \gamma_\beta(1) X(t) \\
    & \ - \int_0^1 K_{vu}(1,y) u(t,y) dy - \int_0^1 K_{vv}(1,y) v(t,y) dy.
\end{split}
\end{equation}
Unlike the existing results in the literature on stabilization of PDE+ODE systems \cite{auriol_delay-robust_2018, di_meglio_stabilization_2018}, %\modifJA{Citer mon papier ODE-EDP, éventuellement celui de Florent} 
we obtain additional terms in the target system. These are due to the additive noise in the SDE that is mapped back into the PDE, through the boundary coupling. Consequently, the states $\alpha$ and $\beta$ are solutions of an SPDE (whose solution should be understood in the weak sense \cite{hairer_introduction_2009}). As it will be seen in the following, this noise cannot be entirely compensated by the control input and will be mapped back into the SDE as additional random drift terms. %The control component $\VBS$ has been chosen to simplify the boundary condition at $x=1$.
Note that, in the absence of the SDE sub-system, the PDE would be finite-time stable. However, due to the cancellation of the term $\rho u(t,1)$, the feedback operator is not strictly proper. This flaw may cause delay/robustness issues as emphasized in \cite{auriol_delay-robust_2018}. \revAns{Partial cancellation of the reflection and filtering methods have been respectively proposed in \cite{auriol_delay-robust_2018} and \cite{auriol_robustification_2023} %\modifJA{Citation} 
to guarantee the existence of robustness margins (with respect to delays and uncertainties). However, these techniques do not directly adjust to the current stochastic configuration. In particular, we should be able to show the robust mean square stability but quantifying the effects of these techniques on the variance seems challenging.  Such a robustness analysis is out of the scope of this paper and left for future works.}% and will be the purpose of future contributions.
%\commentsIB{Il faut être plus précis, robustesse parametrique, robustesse par rapport à la digitalisation}

%\modifJA{J'ai modifié l'ordre des sous-sections}

\subsection{Well-posedness of the closed loop}
%Under appropriate regularity conditions on the control input $\Veff$, we can now show the well-posedness of the closed-loop system 
%\modifJA{J'ai mis condition plutot que assumption, puisque c'est nous qui le choisissons.}
\begin{lemma}\label{lem:well_posed_SPDE_L2}
    If $\Veff$ is in $C^2_\mathcal{F}([0,T])$, then the closed-loop target system~\eqref{eq:first_target_system} is well-posed, that is the processes %\modifJA{Dire ce que ça signifie} 
    $t \mapsto \alpha(t,\cdot)$ and $t \mapsto \beta(t,\cdot)$ are in $C^2_\mathcal{F}([0,T] ; L^2(0,1))$, and the SDE state $X$ is in $C^2_\mathcal{F}([0,T] ; \mathbb{R}^n)$.
\end{lemma}
\begin{proof}
Using the method of characteristics, we have for all $t>\frac{1}{\mu}$ and all $x\in [0,1]$, $
\beta(t,x) = \Veff\left(t - \frac{1-x}{\mu} \right) + \int_{t-\frac{1-x}{\mu}}^t \gamma_\beta \left( x + \mu (t-s) \right) \sigma(s) dW_s.$
Since $\gamma_\beta$ is continuously differentiable, we obtain~\cite[Theorem 6.10]{hairer_introduction_2009} that the stochastic convolution term 
$
t \mapsto \int_{t-\frac{1-\cdot}{\mu}}^t \gamma_\beta \left( \cdot + \mu (t-s) \right) \sigma(s) dW_s,
$
 is sample path continuous in $L^2(0,1)$. % \modifJA{Pourquoi $(0,1)$?}. 
Since $\Veff$ is in $C^2_\mathcal{F}([0,T])$, $t \mapsto \beta(t,\cdot)$ is also sample path continuous in  $L^2(0,1)$.
%Therefore, we have
%\begin{align*}
%\Vert \beta & \Vert_{C^2_\mathcal{F}([0,T] ; L^2(0,1))} \leq \Vert \Veff \Vert_{C^2_\mathcal{F}} \\
%& + \left \Vert \int_{\cdot_t - \frac{1-\cdot_x}{\mu}}^{\cdot_t} \gamma_\beta \left( \cdot_x + \mu (\cdot_t-s) \right) \sigma(s) dW_s \right \Vert_{C^2_\mathcal{F}([0,T] ; L^2(0,1))}.
%\end{align*}
%\modifJA{Cette norme a-t-elle été définie à un moment ?}
The regularity of $\alpha$ can be proven in the same way.
The solution of the SDE system is sample path continuous if the input $t \mapsto \beta(t,0)$ is in $L^2_\mathcal{F}$~\cite{yong_stochastic_1999}.
%This condition on $\beta(\cdot,0)$ cannot be obtained directly from what was proven above, however, u
Using the Burkholder-Davis-Gundy inequality \cite{le_gall_brownian_2016} we have that the stochastic process
$
t \mapsto \int_{t-\frac{1}{\mu}}^t \gamma_\beta \left( \mu (t-s) \right) \sigma(s) dW_s
$
is in $L^2_\mathcal{F}$, which in turn means that $\beta(t,0)$ is in $L^2_\mathcal{F}([0,T])$, which concludes the proof. %\modifJA{N'aurait-il pas fallu écrire ça avant ? Pour pouvoir prendre la norme ?}
\end{proof}
Since the backstepping transformation is a Volterra transformation, it is boundedly invertible \cite{krstic_boundary_2008} %\cite{yoshida_lectures_1960} %\modifJA{Citer Yoshida}
, which implies the well-posedness of the original system~\eqref{eq:coupled_PDE_SDE}.

\subsection{A delayed SDE}
%\modifJA{Je n'aime pas trop le terme explicit solution...}
%We can now rewrite the SDE-SPDE system as a delayed SDE system. More precisely, 
Using the method of characteristics, we have that $X(t)$ is the solution of the following input-delayed stochastic differential equation with random coefficients 
% The explicit solution of $\beta$ can be obtained through the characteristic method and is given by :
% \begin{equation}\label{eq:explicit_formula_beta}
% \beta(t,x) = \Veff\left(t - \frac{1-x}{\mu} \right) + \int_{t-\frac{1-x}{\mu}}^t \gamma_\beta \left( x + \mu (t-s) \right) \sigma(s) dW_s.
% \end{equation}
% By evaluating $\beta$ in $x = 0$ and injecting equation \eqref{eq:explicit_formula_beta} into the SDE, we can see that in order to control the SDE of the original sytem \eqref{eq:coupled_PDE_SDE}, we can focus on controlling the input-delayed stochastic differential equation with random coefficients
\begin{equation}\label{eq:equivalent_delayed_SDE}
\left\{
\begin{array}{l}
     dX(t) = (A X(t) + B \Veff(t-h) + r(t)) dt + \sigma(t) dW_t  \\
     X(0) = X_0, \revAns{\quad \Veff(s) = \beta(0, 1 + \mu s) \  \forall s \in [-h,0)}
\end{array} 
\right.
\end{equation}
where $h \triangleq \frac{1}{\mu}$ and 
$
%r(t) \triangleq B \int_{t-h}^t \gamma_\beta \left( \mu (t-s) \right) \revAns{\mathbbm{1}_{s>0}} \sigma(s) dW_s.
r(t) \triangleq B \int_{t-h}^t \gamma_\beta \left( \mu (t-s) \right) \sigma(s) dW_s.
$

%\modifJA{Je préférerais utiliser U ou $\Veff$ tout le temps pour éviter la surabondance de notations}
%Therefore, if we can design a control input $\Veff$ that minimizes the variance of the delayed SDE \eqref{eq:equivalent_delayed_SDE}, we solve the minimization of the variance for the original problem. 
% the covariance of the original system \eqref{eq:coupled_PDE_SDE} by setting $\Veff = U$ and compute the associated controlled $V_{in}$. 
% Before moving forward with the study of the delayed SDE, we give some regularity properties on the solutions of the target system, which in turn show the well-posedness of the original system.
%From now on, we therefore only focus on controlling the delayed SDE.
Therefore, we now only focus on controlling the delayed SDE~\eqref{eq:equivalent_delayed_SDE} and minimizing the variance of the state. We follow a methodology similar to the one outlined in \cite{velho_mean-covariance_2023}. However, the inclusion of the extra random drift term $r(t)$ requires to modify our approach to minimize the variance, as this term cannot be predicted.

\section{Study of the delayed SDE : a stability result and a minimal bound of the variance}\label{VC}

%In this section, we are interested in providing a minimal bound for the covariance of the state of the SDE \eqref{eq:equivalent_delayed_SDE}, as well as providing an algorithm to minimize it. To do so, we first leverage the Artstein 
In this section, we extend the Artstein transform \cite{artstein_linear_1982} to our stochastic setting to provide a predictor $Y$ of the state $X$ that follows a non-delayed SDE. %\commentsIB{Ajouter une ref parmis celles que nous avons regardé avec Riccardo  } 
%We link the variance of the predictor to the variance of the original state and show the existence of a lower bound on the state covariance under which \eqref{eq:equivalent_delayed_SDE} can not be steered. 
From now on, we denote by $T>h$ a possibly infinite final time.

%In this section, we study which state covariances can be reached by the solution of~\eqref{eq:equivalent_delayed_SDE} within a finite time frame. In particular, we show that there exists a lower bound on the state covariance under which \eqref{eq:equivalent_delayed_SDE} can not be steered. 
%\textcolor{orange}{We first utilize Artstein to transform system into non delayed...  }

\subsection{Artstein transform}

%We achieve the aforementioned goal by leveraging the Artstein transformation \cite{artstein_linear_1982} to our stochastic setting.
Let $X$ be the process solving equation \eqref{eq:equivalent_delayed_SDE}. The Artstein transform of $X$ is the process $Y(t)$, which is adapted to $\mathcal{F}$, that is defined by %$Y$ such that for all $t \in [0,T]$
\begin{equation}\label{eq:def_artstein_transform}
Y(t) \triangleq X(t) + \int_{t-h}^t e^{A(t-s-h)} B \Veff(s) ds. 
\end{equation}
One readily verifies that $Y$ satisfies the following non-delayed stochastic equation 
\begin{equation}\label{eq:artstein_system_linear}
\left\{
    \begin{array}{ll}
    dY(t) &= \hspace{1em} \left( A(t) Y(t) + \overline{B} \Veff(t)  + r(t) \right) dt + \sigma(t) dW_t, \\
    Y(0) &= \hspace{1em} 0 ,
    \end{array}
    \right.
\end{equation}
where $\overline{B} \triangleq e^{-Ah} B$. 
We can now apply known results from non-delayed stochastic control %from \cite{yong_stochastic_1999, liu_optimal_2023} 
to steer \eqref{eq:artstein_system_linear} as desired. However, it is for now unclear how to compute the covariance of $X$, which we aim to estimate once the covariance of $Y$ is available. We have the following lemma
\begin{lemma}
Let $X$ be the process solving equation \eqref{eq:equivalent_delayed_SDE}, and $Y$ its Artstein transform defined in~\eqref{eq:def_artstein_transform}. Then, for all $t \in [h,T]$, we have 
\begin{align}
X(t) = e^{Ah} Y(t-h) &+ \int_{t-h}^t e^{A(t-s)} r(s) ds  \nonumber \\
&+ \int_{t-h}^t e^{A(t-s)} \sigma(s) dW_s .\label{eq:link_artstein_original} 
\end{align}
\end{lemma}
\begin{proof}
Using the analytic formula for linear SDEs~\cite{yong_stochastic_1999}, we obtain $X(t) =  e^{Ah} X(t-h) + \int_{t-h}^t e^{A(t-s)} r(s) ds + \int_{t-h}^t e^{A(t-s)} B \Veff(s-h) ds  + \int_{t-h}^t e^{A(t-s)} \sigma(s) dW_s.$
By expressing $X(t-h)$ in terms of $Y(t-h)$, a change of variable in the regular integral of $\Veff$ yields $X  (t) =  e^{Ah} Y(t-h) + \int_{t-h}^t e^{A(t-s)} r(s) ds 
 - e^{Ah} \int_{t-2h}^{t-h} e^{A(t-s-2h)} B \Veff(s) ds  + \int_{t-2h}^{t-h} e^{A(t-s-h)} B \Veff(s) ds  + \int_{t-h}^t e^{A(t-s)} \sigma(s) dW_s,$
 which gives the desired result. 
\end{proof}
Equation \eqref{eq:link_artstein_original} %provides important insights for the controllability of the process $X$. % In particular, we can leverage it to separate a non-controllable noise term 
states that $X(t)$ can only be controlled through $Y(t-h)$, in that the additional noise term $\int_{t-h}^t e^{A(t-s)} \sigma(s) dW_s$ cannot be controlled. The other noise term can only be partially mitigated by $Y(t-h)$. 

\subsection{Stabilization of the delayed SDE through the Artstein transform.}
If the pair $(A,B)$ is controllable, %\modifJA{stabilizable ne suffirait pas?}
 Artstein's predictor enables the design of a feedback controller to stabilize the delayed SDE in the sense that it ensures exponential convergence to zero of its mean, while keeping the covariance bounded. %Additionally, we can demonstrate that our original system achieves stabilization (in the same sense) via this controller coupled with the backstepping feedback law.

\begin{theorem}\label{thm:stabilization_system_feedback}
Assume that the pair $(A,B)$ is controllable. Define the feedback controller $\Veff(t) = - K Y(t),$
where $Y$ is given by equation~\eqref{eq:def_artstein_transform}, and $K$  is  such that $H \triangleq A - \overline{B}K$ is Hurwitz. Then, the control law $V_{in}(t) = \VBS(t) + \Veff(t)$ drives the means of the states to zero while keeping their variances bounded.
%stabilizes the original system \eqref{eq:coupled_PDE_SDE}. \modifJA{Au sens de la moyenne tout en gardant une variance bornée} 
That is, there exist $C>0$, \revAns{(dependent on the parameters of the system and $K$)} and $\nu>0$ \revAns{(dependent on the eigenvalues of $A-\overline{B}K$)} , such that, for any initial conditions $(X_0,u_0,v_0)$, for all $x$ in $[0,1]$ and $t>0$:
\begin{align*}
   &  \Vert \espE[X(t)] \Vert + \Vert \espE[u(x,t)] \Vert + \Vert \espE[v(x,t)] \Vert \leq C e^{- \nu t} \times\\
   & \quad \times \bigl (\Vert \espE[X_0] \Vert + \Vert  \espE [u_0] \Vert + \Vert  \espE [v_0] \Vert \bigr ), \\
   &  \espE \left[ \Vert X(t)\Vert^2 \right]  + \espE \left[ \Vert u(x,t) \Vert^2 \right] + \espE \left[ \Vert v(x,t) \Vert^2 \right] \leq C
\end{align*}
\end{theorem}
The proof of the Theorem can be found in Appendix \ref{AP}.

%\textcolor{blue}{Faire ici remarque de quelques lignes sur Artstein et limites.}
Despite its theoretical advantages, the Artstein's transform encounters practical limitations. Discretizing the integral during its computation can potentially render the closed-loop system unstable, as discussed in \cite{mondie_finite_2003}.  To address this issue, solutions such as filtering have been proposed to implement the controller safely \cite{mondie_finite_2003,auriol_robustification_2023}. However, exploring these solutions further is beyond the scope of our study. %For stabilizing the system, an interesting workaround involves
Interestingly, setting $\gamma_\beta(0) = K$ in the backstepping transformation, we obtain $dX(t) = [(A-BK)X(t) + r(t)] dt + \sigma(t) dW_t$, which implies the stability if $(A-BK)$ is Hurwitz.

\subsection{Minimum variance bound of the SDE}

We leverage equation \eqref{eq:link_artstein_original} to write $X(t)$ as the sum of a $\mathcal{F}_{t-h}$-measurable process and a stochastic integral depending only on values of $dW_s$ for $s \in [t-h,t]$. The independence of both terms yields a direct relation between the variance of $X(t)$ and of $Y(t-h)$, which is later used in the next section.
%We can leverage equation \eqref{eq:link_artstein_original} in order to separate an $\mathcal{F}_{t-h}$-adapted process from an uncontrollable noise term. This allows us to split the variance of $X(t)$ into a controllable term and a non-controllable term, which results in a minimal variance induced by the delay. \modifJA{Les deux phrases précédentes ne sont pas super claires}. The minimal bound is given by the following lemma
%\modifJA{Attention! Dans ce qui suit les indices de $\gamma_\beta$ ont disparu.}
\begin{lemma}\label{lem:split_artstein_equival_controllable_and_non}
Let $X$ be the process solving equation \eqref{eq:equivalent_delayed_SDE}, and let $Y$ be its Artstein transform~\eqref{eq:def_artstein_transform}. For all $t \in [h,T]$, we have
\begin{align}
V_X(t) = V_{\min}(t) +  &\espE  \left[   [ Y(t-h) + G(t-h) ]^T e^{A^T h} \right. \nonumber \\
&  \left. \times  e^{A h} [ Y(t-h) + G(t-h) ]     \right], \label{eq:eq_split_artstein_transform_control_and_non}
\end{align}
where
\begin{equation*}
\begin{split}
 V_{\min}(t) \triangleq \int_{t-h}^t & \sigma(s)^T \left[ e^{A(t-s)} + N(t-s) \right]^T \\
& \times \left[  e^{A (t-s)} + N(t-s) \right] \sigma(s) ds ,
\end{split}
\end{equation*}
and $G(t) \triangleq \int_{t-h}^t \left[ \int_t^{s+h} e^{A(t-\tau)}\gamma_\beta(\mu (\tau - s) ) d\tau \right] \sigma(s) dW_s $ and $N(u)  \triangleq \int_{-u}^0 e^{-A \tau} \gamma_\beta(\mu( \tau + u)) d\tau$.
%$$
%G(t) \triangleq \int_{t-h}^t \left[ \int_t^{s+h} e^{A(t-\tau)}\gamma_\beta(\tau - s) d\tau \right] \sigma(s) dW_s
%$$
\end{lemma}

\begin{proof} Fix $t \in [h,T]$. We start by separating the integral of $r(s)$ in \eqref{eq:link_artstein_original} into a $\mathcal{F}_{t-h}$-measurable process and a non-controllable process:
\small
\begin{equation*}
\begin{split}
& \int_{t-h}^t e^{A(t-\tau)} r(\tau) d\tau = \int_{t-h}^t \int_{\tau-h}^\tau  e^{A(t-\tau)} \gamma_\beta(\mu( \tau - s )) \sigma(s) dW_s d\tau,
%& = \int_{t-h}^t \int_{t-2h}^t \mathds{1}_{(\tau-h,\tau)}(s) e^{A(t-\tau)} \gamma_\beta(\mu(\tau - s)) \sigma(s)  dW_s d\tau  
\end{split}
\end{equation*}
\normalsize
%where $\mathbf{1}$ is the indicator function. \modifJA{Utiliser $\mathds{1}$ et définir dans la section notation.} 
Since the integrand in the double integral is bounded, we can apply the stochastic Fubini theorem \cite[Chapter 4, Theorem 45]{protter_general_2005} to obtain
\small
\begin{equation*}
\begin{split}
& \int_{t-h}^t e^{A(t-\tau)} r(\tau) d\tau = \int_{t-2h}^{t-h}  \left( \int_{t-h}^{s - h} e^{A(t-\tau)} \gamma_\beta(\mu( \tau - s )) d\tau \right) \times \\
 & \sigma(s) dW_s + \int_{t-h}^{t}  \left( \int_{s}^{t} e^{A(t-\tau)} \gamma_\beta(\mu( \tau - s )) d\tau \right) \sigma(s)  dW_s \\
& \int_{t-h}^t e^{A(t-\tau)} r(\tau) d\tau = e^{Ah} G(t-h) + \int_{t-h}^{t} N(t-s) \sigma(s)  dW_s.
\end{split}
\end{equation*}
\normalsize
%By separating the integral of $r$,
We can then rewrite the state $X$ as the sum of a $\mathcal{F}_{t-h}$ adapted process and a noise:
\begin{equation}\label{eq:X_state_SDE_artstein_separated_drift}
\begin{split}
X(t) = e^{Ah} [ & Y(t-h) + G(t-h) ] \\ 
& + \int_{t-h}^t \left[ e^{A(t-s)}  + N(t-s) \right] \sigma(s) dW_s . 
\end{split}
\end{equation}
We have that $Y(t-h)$ and $G(t-h)$ are $\sigma( W_s: 0 \le s \le t-h ) $-measurable, and that for $r \in [0, t-h]$, $W_r$ is independent from the stochastic integral in equation \eqref{eq:X_state_SDE_artstein_separated_drift} (see details in \cite[Lemma 3]{velho_mean-covariance_2023}). Therefore, $Y(t-h) + G(t-h)$ is independent of the stochastic integral. Consequently, the variance of the sum of the integrals is equal to the sum of the variance of each integral. Using equation~\eqref{eq:quadratic_variation_of_stochastic_integral} for $f_1(s)=f_2(s)=(e^{A(t-s)}  + N(t-s))\sigma(s)$, we obtain that the variance of the stochastic integral corresponds $V_{min}$. This concludes the proof.

% Since the variance of the sum of two independent random variables is the sum of their variances, what is left to prove is that $Y(t-h) + G(t-h)$ is independent of the stochastic integral. This can be done by leveraging the fact that processes $Y(t-h)$ and $G(t-h)$ are $\sigma( W_s: 0 \le s \le t-h ) $-measurable, and that for $r \in [0, t-h]$, $W_r$ is independent from the stochastic integral in equation \eqref{eq:X_state_SDE_artstein_separated_drift} (see details in \cite[Lemma 3]{velho_mean-covariance_2023}). Finally, we obtain 

% From this result, we have 
% \begin{equation*}
% \begin{split}
% & \espE \left[  \left( \int_{t-h}^t \left[ e^{A(t-s)}  + N(t-s) \right] \sigma(s) dW_s \right)^T   \left( \int_{t-h}^t \left[ e^{A(t-s)}\right. \right. \right.   \\
% &  \left.   \left. \left.  + N(t-s) \right] \sigma(s) dW_s \right)   \right] = \int_{t-h}^t \sigma(s)^T \left[ e^{A(t-s)}  + N(t-s) \right]^T \\
% &\left[ e^{A(t-s)}  + N(t-s) \right] \sigma(s) ds
% \end{split}
% \end{equation*}
% which concludes the proof.
\end{proof}
Similarly to \cite{velho_mean-covariance_2023}, we note that the variance of $X$ is bounded by a minimal variance caused by the delay induced by the PDE. 
%In the next section, we give a method to minimize the variance of the other term $Y(t-h) + G(t-h)$ which we can act on with no delay.

%The approach used for the minimization of the variance is similar, as the only difference is the addtional random drift term $r(t)$.

%\textcolor{orange}{Compared to ECC, we have additional term due to ...}

%\textcolor{orange}{Donner plus de précisions sur le terme $r(t)$. Ce n'est pas un terme qui est facile à gérer, bien qu'il soit calculable pour tout temps $t$. Il n'est également pas possible de le prédire. }

%\subsection{Minimizing the variance through LQ optimal control}
%\textcolor{blue}{Better to do it in a separate section?}

\section{Minimization of the variance through LQ control}\label{LQ}

In the previous section, we designed a feedback controller that stabilized in mean the interconnected PDE-SDE system while guaranteeing a bounded variance. However, such a control law may not guarantee that the variance remains small. In what follows, we provide an efficiently implementable control strategy that keeps the state covariance of the delayed SDE consistently low along the trajectory. For this, we propose to compute controllers via techniques from Linear Quadratic (LQ) control \cite{bismut_linear_1976}. In this section, we suppose the final time $T$ is finite. %\modifJA{Reintroduire le $T$. Le prendre suffisament grand. Let us consider $T>0$}.
The goal consists of minimizing the quadratic functional cost $J_R$, defined as:
\begin{equation}\label{eq:def_quadratic_optimal_cost_2}
\begin{split}
    & J_R(\Veff,X) \triangleq  \espE \left[ \int_h^T X(t)^T Q(t) X(t) dt \right] \\
    & \quad + \espE \left[\int_0^{T-h} \Veff(t)^T R(t) \Veff(t) dt \right],
\end{split}
\end{equation}
where $Q \in L^\infty([0,T], \mathcal{S}_n^+)$ and $R \in L^\infty([0,T], \mathcal{S}_n^{++})$. This functional cost penalizes the weighted variance along the trajectory, as well as the control effort. \revAns{These matrices are chosen based on the desired trade-off between state penalization and control effort penalization throughout time.}
%\textcolor{blue}{J'ai décidé de ne pas mettre de coût final pour gagner un peu d'espace. Commentaire?}
Our problem thus states:
\begin{equation}\label{eq:problem_min_cost_X}
    \min_{\Veff \in \mathcal{U}} J_R(\Veff,X) , \quad \text{$X$ solves SDE \eqref{eq:equivalent_delayed_SDE}.}
\end{equation}
Using equation \eqref{eq:link_artstein_original}, we replace the process $X$ with the process $Y + G$ in $J_R$. The benefit of this transformation is that optimal control techniques for non-delayed systems may be leveraged, e.g., \cite{bismut_linear_1976}.

\subsection{Equivalent form of the system.}
In this section, we use a change of variables to write the cost as a linear quadratic cost in terms of a non-delayed SDE. 
Let us introduce the state $\overline Y$ defined by 
\begin{equation}
    \overline{Y}(t) \triangleq Y(t) + G(t),
\end{equation}
where $Y$ has been defined in equation~\eqref{eq:def_artstein_transform}.
For $t > h$, let us denote $V_{min,Q}(t)$ as
\begin{equation*}
\begin{split}
 V_{\min,Q}(t) \triangleq \int_{t-h}^t & \sigma(s)^T \left[ e^{A(t-s)} + N(t-s) \right]^T Q(t) \cdot \\
& \left[  e^{A (t-s)} + N(t-s) \right] \sigma(s) ds .
\end{split}
\end{equation*}
%We can rewrite our minimization problem with the following lemma.
\begin{lemma}\label{lem:rewrite_cost_artstein_nondelayed}
The cost function $J_R$ rewrites as
\begin{equation}\label{eq:def_quadratic_optimal_cost_artstein}
\begin{split}
&J_R(\Veff,X)  = \espE \left[ \int_0^{T-h} \overline{Y}(t)^T \overline{Q}(t) \overline{Y}(t) dt \right] \\ 
& + \espE  \left[ \int_0^{T-h} \Veff(t)^T R(t) \Veff(t) dt \right]  + \int_h^T V_{min,Q}(t) dt,
\end{split}
\end{equation}
where $\overline{Q}(t) \triangleq e^{A^T h} Q(t+h) e^{A h} $ is symmetric positive.
\end{lemma}
\vspace{1em}
\begin{proof}
We can rewrite the term $\espE \left[ X(t)^T Q(t) X(t) \right]$ in equation~\eqref{eq:def_quadratic_optimal_cost_2} using $\bar Y$. Adjusting the computations proposed in Lemma~\ref{lem:split_artstein_equival_controllable_and_non}, we directly obtain 
% By injecting formula \eqref{eq:link_artstein_original} into the term $\espE \left[ X(t)^T Q(t) X(t) \right]$ and using similar computations as the ones in the proof of Lemma \ref{lem:split_artstein_equival_controllable_and_non}, we obtain
% \begin{align*}
% & \espE \left[ X(t)^T Q(t) X(t) \right] = \\
% & \espE \left[ [Y(t-h) + G(t-h)]^T e^{A^T h} Q(t) e^{A h} [Y(t-h) + G(t-h)]^T \right] \\
% & + \espE \left[ \left( \int_{t-h}^{t} \left[  e^{A (t-s)} + N(t-s) \right] \sigma(s) dW_s \right)^T Q(t) \cdot \right. \\
% & \left. \hspace{5em} \left( \int_{t-h}^{t} \left[  e^{A (t-s)} + N(t-s) \right] \sigma(s) dW_s \right) \right].
% \end{align*}
% Therefore, by leveraging \eqref{eq:quadratic_variation_of_stochastic_integral} we may compute
\begin{align*}
\espE \left[ X(t)^T Q(t) X(t) \right] & = \espE \left[ \overline{Y}(t-h)^T \overline{Q}(t-h) \overline{Y}(t-h) \right] \\
& + V_{min,Q}(t).
\end{align*}
This concludes the proof.
\end{proof}
%In order to apply classical LQ techniques, we now only need 
% Let us now write the SDE verified by $\overline{Y}$.
% For improved clarity, we write in what follows
% $$G(t) = \int_{t-h}^t g(t-s) \sigma(s) dW_s $$
% with $g$ a deterministic function defined as
% $$g(u) \triangleq  \int_0^{h-u} e^{-A \tau}\gamma_\beta(\mu (\tau + u )) d\tau $$.
% We note that it is $C^1$.
\begin{lemma}\label{lem:SDE_overline_Y_change_var_artstein}
    The process $\overline{Y}$ satisfies the following SDE
\begin{equation}\label{eq:SDE_overline_Y_change_var_artstein}
\left\{
    \begin{array}{ll}
    d\overline{Y}(t) &= \hspace{1em} \left( A \overline{Y}(t) + \overline{B} \Veff(t)  + \overline{r}(t) \right) dt \\
    & + \ (1 + g(0) ) \sigma(t) dW_t, \\
    \overline{Y}(0) &= \hspace{1em} Y_0 ,
    \end{array}
    \right.
\end{equation}
with $\overline{r}(t) = \int_{t-h}^t \Gamma(t-s) \sigma(s) dW_s$, $\Gamma(u) = B \gamma_\beta(\mu u) + g'(u) - A g(u)$ %\modifJA{J'ai modifié légérement les définitions de $\bar r$ et $\Gamma$ qui me semblaient incorrectes vis à vis de $r(t)$.} 
and $g(u) \triangleq  \int_0^{h-u} e^{-A \tau}\gamma_\beta(\mu (\tau + u )) d\tau$.
\end{lemma}
\begin{proof}
Let us remark that the function $g$ is continuously differentiable and that for all $t>h$, $G(t) = \int_{t-h}^t g(t-s) \sigma(s) dW_s$. Ito's formula yields
\begin{align*}
dG(t) & = g(0) \sigma(t) dW_t - g(h) \sigma(t-h) dW_{t-h} \\ 
& + \left( \int_{t-h}^t g'(t-s) \sigma(s) dW_s \right) dt
\end{align*}
Since $g(h) = 0$, we obtain
\begin{align*}
d\overline{Y}(t) & = dY(t) + dG(t) \\
& = \left( A Y(t) + \overline{B} \Veff(t) + r(t) + A G(t) - A G(t) \right) dt\\
+ &\left( \int_{t-h}^t g'(t-s) \sigma(s) dW_s \right) dt+ (1 + g(0) ) \sigma(t) dW_t
\end{align*}
which finally yields equation \eqref{eq:SDE_overline_Y_change_var_artstein}.
\end{proof}
We can now leverage the results from \cite{bismut_linear_1976} to address the LQ problem expressed in terms of $\overline{Y}$. This approach yields an optimal control solution minimizing \eqref{eq:def_quadratic_optimal_cost_2} in the form of a feedback in $\overline{Y}$. %However, since the SDE verified by $\overline{Y}$ has a random coefficient $\overline{r}$, the optimal control is a solution of a Backward SDE (BSDE) which we solve explicitly in the next subsection.

\subsection{Solving the optimal control problem.}

When addressing a Linear Quadratic (LQ) stochastic control problem, the optimal control is characterized by a solution to a Riccati equation. Since, in our case, the SDE has a random coefficient ($\overline{r}$), the corresponding Riccati equation is a Backward SDE (BSDE). Solving BSDEs is generally complex and requires sophisticated numerical algorithms. However, in our case, we are able to analytically compute an explicit solution for the considered BSDE, allowing a relatively straightforward computation of the optimal control.

\begin{proposition}\label{prop:analytical_solution_LQ_random}
    The optimal control $\Veff^*$ minimizing \eqref{eq:def_quadratic_optimal_cost_2} is given by
\begin{equation}\label{eq:form_of_optimal_control_LQ_random_drift}
    \Veff^*(t) = - R^{-1} \overline{B}^T  \left(  P(t) \overline{Y}(t) + \phi(t)  \right),
\end{equation}
where $P$ is a symmetric matrix solution to the deterministic Riccati ODE
\begin{align*}
    \dot{P}(t) & = -A^T P(t) - P(t)  A - \overline{Q}(t)
 + \ P(t)  \overline{B}  R^{-1} \overline{B} ^T P(t),
\end{align*}
% \begin{equation}\label{eq:Riccati_ODE_optimal}
%     \left\{
%     \begin{array}{ll}
%     \dot{P}(t) & = -A^T P(t) - P(t)  A - \overline{Q}(t) \\
%     & + \ P(t)  \overline{B}  R^{-1} \overline{B} ^T P(t) \\
%     P(T-h) &= 0
%     \end{array}
%     \right.
% \end{equation}
with $P(T-h)=0$, and where  $\phi$ is given by
\begin{equation*}%\label{eq:optimal_openloop_control_part}
    \phi(t) = \int_{t-h}^t \left[  \int_t^{s+h} \Phi_\Pi(t,\tau)P(\tau)\Gamma(\tau - s) d\tau   \right] \sigma(s) dW_s,
\end{equation*}
with $\Phi_\Pi(\cdot, \cdot) $ the fundamental matrix associated to the matrix function $\Pi(\cdot)$ defined as $
\Pi(t) \triangleq P(t) \overline{B} R^{-1} \overline{B}^T - A.$
\end{proposition}
\begin{proof}
    Since $V_{min,Q}(t)$ is independent from the control variable, Lemma \ref{lem:rewrite_cost_artstein_nondelayed} shows that the minimization problem \eqref{eq:problem_min_cost_X} is equivalent to the following:
\begin{equation}\label{eq:problem_min_cost_Y}
    \min_{\Veff \in \mathcal{U}} \overline{J}_R(\Veff,\overline{Y}) , \quad \text{$\overline{Y}$ solves SDE \eqref{eq:SDE_overline_Y_change_var_artstein}.}
\end{equation}
where $\overline{J}_R\triangleq  \espE \left[ \int_0^{T-h} \overline{Y}(t)^T \overline{Q}(t) \overline{Y}(t) dt  \right]  + \espE \left[ \int_0^{T-h} \Veff(t)^T R(t) \Veff(t) dt \right] .$ To solve this LQ problem, we may apply the results in \cite{bismut_linear_1976} to obtain that the optimal control $\Veff^*$ is given by \eqref{eq:form_of_optimal_control_LQ_random_drift}, with $P$ verifying the BSDE
\begin{equation*}
    \left\{
    \begin{array}{ll}
    dP(t) & = \left[ -A^T P(t)  - P(t)  A - \overline{Q}(t) \right. \\
    & \ + \left. P(t)  \overline{B}  R^{-1} \overline{B} ^T P(t) \right] dt + \Lambda(t) dW_t,
    \end{array}
    \right.
\end{equation*}
with $P(T-h) = 0$, and with $\phi$ verifying the BSDE
\begin{equation*}\label{eq:Stochastic_Riccati_BSDE_optimal_drift_phi}
    \left\{
    \begin{array}{ll}
    d\phi(t) & = \left[ \Pi(t) \phi(t) - P(t) \overline{r}(t) \right]dt + \xi(t) dW_t, 
    \end{array}
    \right.
\end{equation*} with $\phi(T-h)=0$.
The BSDE verified by $P$ is only composed of deterministic coefficients, and therefore $\Lambda = 0$.
As for the BSDE verified by $\phi$, it is linear affine and can, therefore, be solved explicitly. Using the method described in \cite[Chapter 10.2]{touzi_introduction_2013}, we obtain
\begin{equation}\label{eq:conditional_expectation_sol_BDSE_phi_riccati}
\phi(t) = \espE \left[ \left. \int_t^{T-h} \Phi_\Pi(t, \tau) P(\tau) \overline{r}(\tau) d\tau \right\vert \mathcal{F}_t \right].
\end{equation}
by writing $\overline{r}$ as a stochastic integral, the stochastic Fubini theorem yields 
\footnotesize
\begin{equation*}
\phi(t) = \espE \left[ \left. \int_{t-h}^{T-h} \left( \int_{\min(t,s)}^{\max(T-h, s+h)} \Psi(t,\tau,s) d\tau \right) \\ \sigma(s) dW_s \right\vert \mathcal{F}_t \right]
\end{equation*}
\normalsize
where $
\Psi(t,\tau,s) \triangleq \Phi_\Pi(t, \tau) P(\tau) \Gamma(\tau - s).$
%\modifJA{A vérifier avec la def de $\bar r$.}
We conclude the proof using the classical result \cite{le_gall_brownian_2016} that for any function $f \in L^2_\mathcal{F}(0,T)$ and any $0 \leq t_1 \leq t_2 \leq T$,
$
\espE \left[ \left. \int_{0}^{t_2} f(s) dW_s \right\vert \mathcal{F}_{t_1} \right] = \int_{0}^{t_1} f(s) dW_s.
$
\end{proof}

\revAns{In simulations, computing the stochastic integral is straightforward as we have access to the random realizations of the Brownian motion that is sampled. However, in practice, obtaining the Brownian motion trajectory may not always be feasible. We can address this challenge through Monte Carlo methods. A common approach in stochastic control is to precompute numerous optimal trajectories through simulation and store them. When employing the algorithm online, users can then reference these pre-generated optimal trajectories.}

%The computation of this optimal control is straightforward in the sense that it does not require numerical solutions for a BSDE or conditional expectations. However, it does require the computation of a stochastic integral involving $\sigma(t)dW_t$. Accessing this term might pose challenges in practical scenarios. One possible solution is to approximate the conditional expectation \eqref{eq:conditional_expectation_sol_BDSE_phi_riccati} through a Monte-Carlo approach.  
%Another workaround is to estimate the noise $\sigma(s) dW_s$ through the input of the SDE $v(t,0)$ that is typically accessible.
%Nonetheless, in our case, we typically have access to the input of the Stochastic Differential Equation (SDE), $v(t,0)$. Consequently, we can consistently obtain $\int_{t-h}^t \gamma_\beta(\mu (t-s )) \sigma(s) dW_s$, which can in turn be used to estimate the noise. 

%In both cases, a study of the effects of the error of approximation will be the purpose of future research.

%should be done to ensure the relevance of this approach. % We do not delve more into the subject as it is out of the scope of the study.
%\modifJA{A changer}

\section{Simulations}\label{SI}

%In this section, we present simulation results applying our feedback stabilization method to an unstable one-dimensional SDE. The system parameters are set as follows: $A = 1.5$, $B = 1$, $\sigma(t) = 1$, $\mu = \lambda = 3$, $\eta^{+} = \eta^{-} = 0.2$, $M = \rho = q = 0.5$, $T = 4$, and $X_0 = 2$.

\revAns{ In this section, we present simulation results applying our feedback stabilization method presented in Theorem \ref{thm:stabilization_system_feedback} to an unstable one-dimensional SDE. This unstable system is academic, the parameters are chosen to demonstrate the efficiency of our method. The SDE parameters are chosen as follows : $A = 0.6, B = 1$, $\sigma(t) = 0.6$ and $X_0 = 2$. The PDE  parameters are : $\mu = 2$, $\lambda = 1$, $\eta^{+} = \eta^{-} = 0.3$, $M = \rho = 1$ and $q = 0.25$. }

\begin{figure}[ht!]
\centering
  \includegraphics[width=0.90\linewidth]{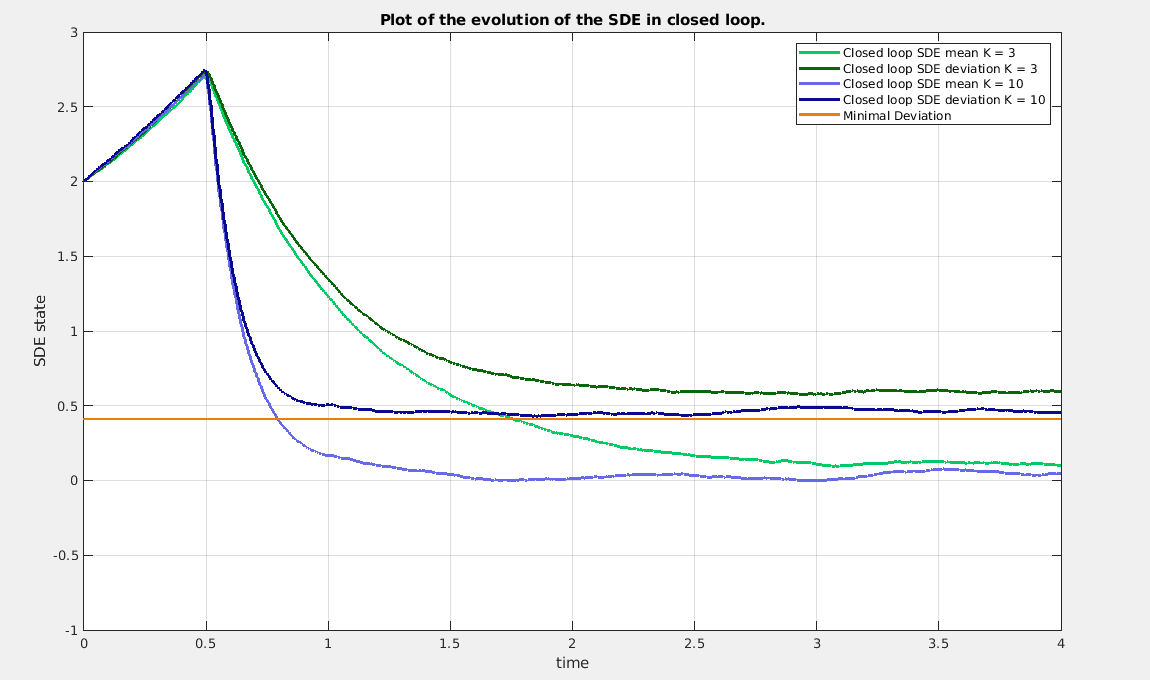}
  \caption{\centering Mean and deviation (square root of the variance) of the SDE state are depicted two systems stabilized with different gains. The values were computed through a Monte-Carlo approach.}% \textcolor{red}{mettre LQ et non pas LQT}}
  \label{fig:plot_SDE_mean_var_closed_open}
\end{figure}

%Results in Figure \ref{fig:plot_SDE_mean_var_closed_open} demonstrate that, even though the delay caused by the PDE induces a short spike at the start, the controller effectively stabilizes the state while maintaining bounded deviation. 

\revAns{  Results in Figure \ref{fig:plot_SDE_mean_var_closed_open} demonstrate that, even though the delay caused by the PDE induces a short spike at the start, the controller effectively stabilizes the state while maintaining bounded deviation. Moreover, a higher gain in the controller induces a lower variance of the state and therefore a safer control. }

\section{Conclusion}
In this paper, we have introduced two approaches for the control of bidirectionally coupled PDE+SDE systems. These two approaches utilize a combination of classical PDE+ODE tools and stochastic control techniques. One approach achieves stabilization in mean with bounded variance of the states, while the other minimizes the SDE variance to ensure additional robustness against noise.
We plan to explore further avenues for research, including the generalization of the PDE to a system of $n \times m$ linear first-order equations, as well as the extension of the SDE framework to incorporate multiplicative noise. \revAns{We also plan to investigate the potential benefits of employing alternate backstepping techniques used successfully in PDE+ODE systems, such as the multi-step approach \cite{irscheid2023output}.}

%$$\beta(t,0)=\rho q \beta(t-\tau,0)+V_e(t-\frac{1}{\mu})+\int \Gamma dW_s$$

%$$Z(t)=X(t)-\rho q X(t-\tau)$$

%$$dZ=(AZ+BV_e(t-h)+r_1(t))dt+\sigma dW$$

\bibliographystyle{ieeetr}
\bibliography{references_cropped}

%\newpage

\section{Appendix - Proof of Theorem \ref{thm:stabilization_system_feedback}}\label{AP}

In what follows, if $X$ is a stochastic process, we denote its expectation by $m_X(t) \triangleq \espE[ X(t) ]$. If $u$ is a function of time and space, we denote its expectation by $m_u(t,x) = \espE[ u(x,t) ]$. We denote by $C$ a possibly overloaded constant dependent on the parameters of the system.

\begin{proof}
    Let $(u, v, X)$ be the solution of the dynamics (2), and $(\alpha, \beta)$ given by the backstepping transformation (3), and $Y$ the Artstein transform of $X$. We consider a controller $V_{in} = V_{BS} + V_{eff}$ with the backstepping controller $ V_{BS}$ given by (6) and the effective SDE controller $V_{eff}$ given by $V_{eff}(t) = - K Y(t)$. The gain $K$ is chosen such that $H \triangleq A - \overline{B} K$ is Hurwitz, which is possible since if $(A,B)$ is controllable, then $(A,\overline{B})$ is controllable (see \cite[Theorem 6]{velho_mean-covariance_2023}).

    \textbf{First part of the proof: mean exponential stabilization.}
    With this controller, for $t > 0$, $Y$ follows the dynamic 
    \begin{equation*}
        dY(t) = \bigl ( H Y(t) + r(t) \bigr ) dt + \sigma(t) dW_t. 
    \end{equation*}
    where $r(t)$ is given by the stochastic integral as in equation (7). We note that it is therefore always of mean 0. The mean $m_Y(t)$ is therefore given by
    $$
    m_Y(t) = e^{H t} m_Y(0).
    $$
    At the initial time, the process $Y$ is given by 
    \begin{equation*}
    \begin{split}
        Y(0) & = X(0) + \int_{-h}^0 e^{A(-s-h)} B \Veff(s) ds \\
        & = X_0 + \int_{-h}^0 e^{A(-s-h)} B \beta(0, 1 + \mu s) ds.
    \end{split}
    \end{equation*}
    Through the definition of the backstepping transformation, since the kernels $K$ are bounded, we have that 
    $$
    \Vert \beta(0,x) \Vert \leq \Vert v(0,x) \Vert + C \Bigl ( \Vert X_0 \Vert + \Vert u_0 \Vert_{L^2} + \Vert v_0 \Vert_{L^2} \Bigr )
    $$
    Therefore 
    $$
    \Vert m_Y(t) \Vert \leq C e^{- \nu t} \Bigl ( \Vert X_0 \Vert + \Vert u_0 \Vert_{L^2} + \Vert v_0 \Vert_{L^2} \Bigr )
    $$
    Where $\nu$ is the highest eigenvalue of $H$. By taking the expectation of (10), we obtain
    $$
    m_X(t) = e^{A h } m_Y(t-h),
    $$
    and by taking the expectation of the explicit value of $\beta(t,x)$ (as in the proof of Lemma 1), we obtain 
    \begin{align*}
    m_\beta & (t,x) = \espE \biggl [  \Veff \left(t- \frac{1-x}{\mu} \right) \biggr . \\
    & \biggl . \hspace{5em} + \int_{t-\frac{1-x}{\mu}}^t \gamma_\beta \left( x + \mu (t-s) \right) \sigma(s) dW_s  \biggr] \\
    & = \espE \left[ \Veff \left(t- \frac{1-x}{\mu} \right) \right] = - K m_Y \left(t- \frac{1-x}{\mu} \right).
    \end{align*}
    We can obtain a similar equation for $\alpha$.
    Therefore, the expectations of $Y(t)$, $X(t)$, $\alpha(t,x)$ and $\beta(t,x)$ are exponentially decreasing. We conclude by bounding in the exact same way the expectation of $u$ and $v$ through the linear inverse backstepping transformation 
    \begin{align*}
    \begin{pmatrix}
        u(t,x)\\v(t,x)
    \end{pmatrix}=\begin{pmatrix}
        \alpha(t,x)\\\beta(t,x)
    \end{pmatrix}&+\int_0^x L(x,y)\begin{pmatrix}
        \alpha(t,y)\\\beta(t,y)
    \end{pmatrix}dy \nonumber \\   &+\gamma_L(x)X(t),\label{eq:backstepping_transform_alpha_beta_definition}
\end{align*}
where $L$ is a bounded kernel.

    \textbf{Second part of the proof: bound on the variance.}
    We need first to bound the variance of $Y(t)$. We do this through its explicit formula:
    $$
    Y(t) = e^{Ht} Y(0) + \int_0^t e^{H(t-s)} \sigma(s) dW_s + \int_0^t e^{H(t-s)} r(s) ds
    $$
    therefore, by expanding the product and using Young inequality on the crossed terms we obtain
    \begin{align*}
        & \espE[ Y(t)^T Y(t) ] \leq C \Bigl ( \espE \left [ \left \Vert e^{Ht} Y(0) \right \Vert^2 \right ] + \int_0^t \left \Vert e^{H(t-s)} \sigma(s) \right \Vert^2 ds \\
        & \hspace{5em} +  \int_0^t e^{H^T(t-s)} \espE [ r(s)^T r(s) ] e^{H(t-s)} ds \Bigr ).
    \end{align*}
    Thanks to the Itô formula for the quadratic variation of a stochastic integral, the expectation term in $r$ can be bounded as follows
    $$
    \espE [ r(s)^T r(s) ] \leq h \Vert \sigma \Vert_\infty^2 \Vert \gamma \Vert_\infty^2 \leq C.
    $$
    Since $H$ is Hurwitz and all the other terms in the integrals are bounded, we have that
    $$
    \espE[ Y(t)^T Y(t) ] \leq C.
    $$
    with C depending on the norms of the parameters of the system.
    From this bound, we can deduce in the same way that $\espE[ X(t)^T X(t) ] \leq C$ by using (10). The bound on the variance of $\alpha(t,x)$ and $\beta(t,x)$ is deduced from their explicit expression (always using Young inequality for the crossed terms, then the Ito formula for the stochastic integrals)
    
    \begin{align*}
        & \espE[ \beta(t,x)^T \beta(t,x) ] \leq  C h \Vert \sigma \Vert_\infty^2 \Vert \gamma \Vert_\infty^2\\
        & \quad + C \Vert K \Vert^2 \espE \left[ Y\left(t- \frac{1-x}{\mu} \right)^T Y\left(t- \frac{1-x}{\mu} \right) \right]  
    \end{align*}
    Finally, the bound on $u(t,x)$ and $v(t,x)$ can be deduced from the inverse backstepping transformation and the fact that the kernels are bounded.

\end{proof}

\end{document}